\documentclass{amsart}

\newtheorem{theorem}{Theorem}
\newtheorem{lemma}{Lemma}

\newcommand{\C}{\mathbb C}
\newcommand{\N}{\mathbb N}
\newcommand{\D}{\mathbb D}

\title{Density of polynomials in classes of functions on products of planar domains}

\author{P. M. Gauthier and V. Nestoridis}

\address{D\'epartement de math\'ematiques et de statistique, Universit\'e de Montr\'eal,
CP-6128 Centreville, Montr\'eal,  H3C3J7, CANADA}
\email{gauthier@dms.umontreal.ca}

\address{Department of Mathematics, University of Athens Panepisitmioupolis, 15784 Athens, Greece}
\email{vnestor@math.uoa.gr}


\thanks{Research supported by NSERC (Canada)}

\begin{document}

\maketitle

\begin{abstract}
We give sufficient conditions on planar domains $\Omega_i, i\in I,$ where $I$ is an arbitrary set, so that polynomials are dense in $A(\prod_{i\in I}\Omega_i)$ as  well as in $A^\infty(\prod_{i\in I}\Omega_i),$ endowed with their natural topologies.  We also characterize the uniform limits, with respect to the chordal metric, of polynomials on $\prod_{i\in I}\overline\Omega_i.$
\end{abstract}

\section{Introduction}

In \cite{MeN} and \cite{KKN} the space $A^\infty(\Omega)$ was used in order to establish universality of Taylor series in the sense of Luh \cite{L} and Chui and Parnes \cite{CP}. This universality has been proven to be generic, in the closure of polynomials $X^\infty(\Omega)$ in $A^\infty(\Omega),$ with its natural topology. It is an open question whether $X^\infty(\Omega)=A^\infty(\Omega)$ always holds for every planar domain $\Omega,$ such that $\overline\Omega^o=\Omega$ and $[\C\setminus\overline\Omega]\cup\{\infty\}$ is connected. In \cite{NZ} a sufficient condition of geometric nature is given assuring $X^\infty(\Omega)=A^\infty(\Omega).$ The condition is the existence of an $M<+\infty,$ such that every two points in $\Omega$ can be joined in $\Omega$ by a path of length at most $M.$

In Section 3 we extend the previous result to arbitrary products $\Omega=\prod\{\Omega_i:i\in I\},$ where $\Omega_i\subset\C$ is a bounded domain such that $(\overline\Omega_i)^o=\Omega_i,$ $\overline\Omega^c$ is connected and there exist $M_i<+\infty,$ with the property that every two points in $\Omega_i$ may be joined by a path in $\Omega_i$ of length at most $M_i.$ We shall denote the products $\prod\{\Omega_i:i\in I\}$ and  $\prod\{\overline\Omega_i:i\in I$ respectively by $\Omega$ and $\overline\Omega.$

Under the above assumptions, a function $f:\overline\Omega\rightarrow\C$ belongs to the class $A(\Omega),$ if it is continuous on $\overline\Omega$ endowed with the product topology and separately holomorphic in $\Omega.$ We endow $A(\Omega)$ with the supremum norm and we show that, under the above assumptions, the polynomials are dense in $A(\Omega).$ We notice that, even in the case where $I$ is infinite, each polynomial depends on only finitely many variables.

Further, for every function $f\in A(\Omega),$ the function $Df$ is well-defined and separately holomorphic on $\Omega,$ for every differential operator $D$ of the form
$$
(*) \quad \quad \quad \quad D = \frac{\partial^{\alpha_1+\cdots+\alpha_n}}{\partial z_{i_1}^{\alpha_1}\cdots\partial z_{i_n}^{\alpha_n}},
$$
where $i_1,\cdots,i_n\in I, 0\le\alpha_i<+\infty, \alpha_i\in\N,$ for all $i=1,\cdots,n$ and $n\in\N.$

{\bf Definition.} Under the above assumptions and notations, we define $A^\infty(\Omega)$ to be the class of functions $f\in A(\Omega)$ such that $Df\in A(\Omega),$ for every differential operator $D$ satisfying (*). We endow $A^\infty(\Omega)$ with the topology of uniform convergence of every mixed derivative given by an operator $D$ of the form (*). The main theorem proven in Section 3 is the following.

\begin{theorem}\label{polynomials dense}
Under the above assumptions and notations, the polynomials are dense in $A^\infty(\Omega).$
\end{theorem}

In Section 4, we consider an arbitrary family of Jordan domains $G_i\subset\C, i\in I,$ where $I$ is an arbitrary set. Let $G=\prod\{G_i:i\in I\}$ and $\overline G=\prod\{\overline G:i\in I\}$ be endowed with the product topology. We investigate the class $\widetilde A(G)$ of uniform limits of polynomials on $\overline G,$ with respect to the chordal distance $\chi.$  We prove the following.

\begin{theorem}\label{chordal}
Under the above assumptions and notations, $\widetilde A(G)$ contains exactly the constant function equal to $\infty$ and all continuous functions $f:\overline G\rightarrow\C\cup\{\infty\},$ where $\overline G$ is endowed with the product topology, such that $f(G)\subset\C$ and $f|_G$ is separately holomorphic.
\end{theorem}

For an example of an $\Omega$ satisfying the hypotheses of Theorem \ref{polynomials dense} but not those of Theorem \ref{chordal}, it is sufficient to give domains $\Omega_i$'s for Theorem \ref{polynomials dense} which are not Jordan domains. Consider the domain
$$
    \{x+iy:0<x<1, -5<y<\sin(1/x)\}
$$
Setting each $\Omega_i$ equal to this domain, we have that the $\Omega_i$ satisfy the assumptions of Theorem \ref{polynomials dense}, but they are not Jordan domains.

Theorem \ref{chordal} extends results from \cite{BGH}, \cite{N}, \cite{N'}, \cite{FNP} and \cite{MN}.


\section{Preliminaries}

\begin{lemma}\label{R(K)=A(K)}
If $K_1,\cdots,K_m$ are compact sets in $\C$ such that $R(K_i)=A(K_i), i=1,\cdots,m,$ and $K=K_1\times\cdots\times K_m,$ then $R(K)=A(K).$
\end{lemma}

\begin{proof}
If $m=1,$ there is nothing to prove. Suppose $m>1.$
The last corollary in \cite{GG} states that if $X$ is a compact subset of $\C$ such that $R(X)=A(X),$ and if $Y$ is a compact subset of $\C^n$ such that $R(Y)=A(Y),$ then $R(X\times Y)= A(X\times Y).$  Setting $X=K_{m-1}$ and $Y=K_m,$ we conclude that $R(K_{m-1}\times K_m)=A(K_{m-1}\times K_m).$ If $m=2,$ we are through. If not, applying the corollary again, with $X=K_{m-2}, Y=K_{m-1}\times K_m,$ we obtain
$$
    R(K_{m-2}\times K_{m-1}\times K_m) =  R(K_{m-2}\times (K_{m-1}\times K_m)) =
$$
$$
    A(K_{m-2}\times (K_{m-1}\times K_m)) =  A(K_{m-2}\times K_{m-1}\times K_m).
$$
After finitely many steps we arrive at the desired conclusion.
\end{proof}

\begin{lemma}\label{P(K)=A(K)}
If $K_1,\cdots,K_m$ are compact sets in $\C$  and $K=K_1\times\cdots\times K_m,$ then $P(K)=A(K),$ if and only if  $K_i^c$ is connected, for each $i=1,\cdots,m.$
\end{lemma}

\begin{proof}
Suppose $K_i^c$ is connected, for each $i=1,\cdots,m.$
Let $f\in A(K)$ and $\varepsilon>0.$ By Mergelyan's theorem $P(K_i)=A(K_i)$ and so by Lemma \ref{R(K)=A(K)}, there is a rational function $R$ whose singularities lie outside $K$ such that $|R-f|<\varepsilon/2$ on $K.$  Since the singularities form a closed set, they are at a positive distance from $K$ and so $R$ is holomorphic on $K.$ Since $K$ is polynomially convex, it follows from the Oka-Weil Theorem that there is a polynomial $P$ such that $|P-f|<\varepsilon$ on $K,$ which concludes the proof in one direction.

The other direction is obvious by considering functions depending on only one variable and seeing them as elements of $A(K).$
\end{proof}


\section{Proof of Theorem \ref{polynomials dense}}

Let $\Omega_i\subset\C, i=1,\cdots,m$ be bounded domains such that $\overline\Omega_i^c$ are connected and $(\overline\Omega_i)^o=\Omega_i.$ Set $\Omega=\Omega_1\times\cdots\Omega_m.$
Let $g\in A(\Omega)= A(\Omega_1\times\cdots\times\Omega_m)$ and $\epsilon>0.$ Then, by Lemma \ref{P(K)=A(K)} there exists a polynomial $Q(z_1,\cdots,z_m)$ such that $|g-Q|<\epsilon$ on $\overline\Omega_1\times\cdots\times\overline\Omega_m.$

Suppose there is an $M<+\infty,$ such that for all $i=1,\cdots,m$ and for all $z,w\in\Omega_i,$ there exists a path $\gamma$ joining $z$ and $w$ in $\Omega_i,$ whose length is at most $M.$ We shall prove that the polynomials are dense in $A^\infty(\Omega_1\times\cdots\times\Omega_m).$ That is, given $f\in A^\infty(\Omega_1\times\cdots\times\Omega_m), \varepsilon>0$ and $n\in\N,$ we shall find a polynomial $P(z_1,\cdots,z_m)$ such that
$$
   \left|\frac{\partial^{\alpha_1+\cdots+\alpha_m}P}{\partial z_1^{\alpha_1}\cdots\partial z_m^{\alpha_m}} -
    \frac{\partial^{\alpha_1+\cdots+\alpha_m}f}{\partial z_1^{\alpha_1}\cdots\partial z_m^{\alpha_m}}\right| < \varepsilon,
    \quad 0\le\alpha_i\le n.
$$

Consider
$$
    \frac{\partial^{nm}}{\partial z_1^n\cdots\partial z_m^n}f.
$$
Since Mergelyan's theorem holds on $\overline\Omega_1\times\cdots\overline\Omega_m,$ we find a polynomial $Q(z_1,\cdots,z_m)$ such that
\begin{equation}\label{Q-frac}
    \left|Q(z_1,\cdots,z_m)-\frac{\partial^{nm}}{\partial z_1^n\cdots\partial z_m^n}f(z_1,\cdots,z_m)\right| < \varepsilon \quad
    \mbox{on} \quad \overline\Omega_1\times\cdots\overline\Omega_m.
\end{equation}

Without loss of generality, we may suppose that $0\in\Omega_i, i=1,\cdots,m$ and $M\le 1.$ For a function $E(z_1,\cdots,z_m)$ holomorphic in $\Omega_1\times\cdots\times\Omega_m,$ we define a multiple integral operator  $T[E],$ by integrating with respect to each variable $n$ times, each time starting at $0,$  while keeping the other variables fixed. The integrals are along paths starting at $0$ and of length less than $1.$ It follows that
\begin{equation}\label{T[E]}
    |T[E]|\le\sup|E|.
\end{equation}

Set
$$
    T[Q] = A \quad \mbox{and} \quad  T\left[\frac{\partial^{nm}}{\partial z_1^n\cdots\partial z_m^n}f\right] = B.
$$
Then, from (\ref{Q-frac}) and (\ref{T[E]}), we have
\begin{equation}\label{A-B}
    \left|T\left[Q-\frac{\partial^{nm}}{\partial z_1^n\cdots\partial z_m^n}f\right]\right| =
    \left|T[Q]-T\left[\frac{\partial^{nm}}{\partial z_1^n\cdots\partial z_m^n}f\right]\right| = |A-B| < \varepsilon,
\end{equation}
For example, for $n=m=2,$ we have
$$
    T[E](z_1,z_2) = \int_{b_2=0}^{z_2}\int_{b_1=0}^{b_2}\int_{a_2=0}^{z_1}\int_{a_1=0}^{a_2}E(a,b)da_1da_2db_1db_2
$$
and, in particular,
$$
    B(z_1,z_2) = T\left[\frac{\partial^4}{\partial z_1^2\partial z_2^2}f\right](z_1,z_2) =
$$
$$
    T[D^2_1D^2_2f](z_1,z_2) =  \int_{b_2=0}^{z_2}\int_{b_1=0}^{b_2}\int_{a_2=0}^{z_1}\int_{a_1=0}^{a_2}(D^2_1D^2_2f)(a,b)da_1da_2db_1db_2.
$$

The function $A(z_1,\cdots,z_m)$ is a polynomial, since $Q$ is a polynomial. The function $B(z_1,\cdots,z_m)$ is a finite sum, where one term  is $f(z_1,\cdots,z_m).$ The remaining terms are products of polynomials and mixed derivatives of $f$ evaluated at points $(w_1,\cdots,w_m),$ where $w_i=0$ or $w_i=z_i,$ but there is at least one $i$ where $w_i=0.$  Thus, each of these remaining terms is the product of a polynomial and an element of $A^\infty(\prod_{i\in S}\Omega_i),$ where $S\subset\{1,\cdots,m\}$ contains less than $m$ elements. Relation (\ref{A-B}) can be written as
\begin{equation}\label{F}
    |F-f| < \varepsilon \quad \mbox{on} \quad \overline\Omega_1\times\cdots\overline\Omega_m, \quad \mbox{where} \quad  F\equiv A-B+f.
\end{equation}

Remark. Since each term of $F$ is the product of a polynomial and an $A^\infty$-function depending on strictly less than $m$ of the  variables $z_1,\cdots,z_m,$  it follows by induction that $F$ and any finite set of its partial derivatives can be uniformly  approximated by a polynomial and respectively its corresponding partial derivatives on $\overline\Omega_1\times\cdots\times\overline\Omega_m.$ See \cite{NZ} for the case $m=1,$ although it is not really needed.
Hence, it is sufficient to show that, for all $0\le\alpha_1\le n, i=1,\cdots,m,$ we have
$$
    \left|\frac{\partial^{\alpha_1+\cdots+\alpha_m}}{\partial z_1^{\alpha_1}\cdots\partial z_m^{\alpha_m}}F -
    \frac{\partial^{\alpha_1+\cdots+\alpha_m}}{\partial z_1^{\alpha_1}\cdots\partial z_m^{\alpha_m}}f\right| < \varepsilon,
$$
which is equivalent to
\begin{equation}\label{partial(A-B)}
    \left|\frac{\partial^{\alpha_1+\cdots+\alpha_m}}{\partial z_1^{\alpha_1}\cdots\partial z_m^{\alpha_m}}(A-B)\right| < \varepsilon.
\end{equation}
But $A-B = T[\Gamma],$ where
$$
    \Gamma = Q(z_1,\cdots,z_m)-\frac{\partial^{nm}}{\partial z_1^n\cdots\partial z_m^n}f(z_1,\cdots,z_m)
    \quad \mbox{and} \quad |\Gamma|<\varepsilon.
$$
Thus, we find that the left member of (\ref{partial(A-B)}) is a multiple integral of $\Gamma,$ where we integrate $\Gamma$ $n-\alpha_i$ times with respect to each variable $z_i$ . Since the paths of integration are of length at most $M\le 1$ and $|\Gamma|<\varepsilon,$ we have the estimate (\ref{partial(A-B)}). This concludes the proof.

Now, we shall extend the latter result to an arbitrary (not necessarily countable) number of complex variables.

Let $I$ be an infinite set and let $\Omega_i, i\in I,$ be a family of bounded domains in $\C,$ such that $(\overline\Omega_i)^o=\Omega_i,$ $\overline\Omega_i^c$ is connected and there exist $M_i<+\infty,$ such that every pair of points $z$ and $w$ in $\Omega_i$ can be joined by a path in $\Omega_i$ of length less than $M_i.$  Set $\Omega=\prod\{\Omega_i:i\in I\}$ and consider the spaces of functions $A(\Omega)$ and $A^\infty(\Omega).$ The space $A(\Omega)$ consists of all functions
$$
    f:\prod\{\overline\Omega_i:i\in I\} \rightarrow \C
$$
continuous in the product topology, such that if we fix all of the variables but one, $f$ is a holomorphic function of that variable in the corresponding $\Omega_i.$ For such a function $f$ and for an arbitrary $\varepsilon,$ one can find a finite subset $F\subset I,$ such that for every $\zeta=(\zeta_i)_{i\in I}\in\prod_{i\in I}\overline\Omega_i,$ if we put $w(F,\zeta,z) = (w_i)_{i\in I},$ with
$$
     w_i =
    \left\{
    \begin{array}{lll}
     z_i & \mbox{for} & i\in F\\
     \zeta_i & \mbox{for} & i\in I\setminus F,
    \end{array}
    \right.
$$
by \cite{MN}, we have $|f(z)-f(w(F,\zeta,z)|<\varepsilon/2.$ For $F$ and $\zeta$ fixed, the function $z\mapsto f(w(F,\zeta,z)$ is essentially a function depending on a finite number of variables. Invoking the preceding results, we find a polynomial $P,$ depending on a finite number of variables, such that $|f(w(F,\zeta,z)-P(z)|<\varepsilon/2,$ whence $|f-P|<\varepsilon.$ We thus see that the polynomials are dense in $A(\Omega),$ for the topology of uniform convergence on $\prod_{i\in I}\overline\Omega_i.$ One sees immediately that the reciprocal is also true. Thus, $A(\Omega)$ is precisely the set of uniform limits on $\prod_{i\in I}\overline\Omega$ of polynomials. We notice that each polynomial depends on a finite set (depending on the polynomial) of variables.

Consider a function $f:\Omega\rightarrow\C,$ which is holomorphic with respect to each variable separately. Then, the mixed partial derivatives, with respect to a finite set of indices $F\subset I$ are well-defined on $\Omega.$ Such a function belongs to $A^\infty(\Omega)$ if all mixed partial derivatives extend to $\prod_{i\in I}\overline\Omega_i$ continuously with respect to the product topology; in other words, $f\in A^\infty(\Omega)$ if and only if all mixed partial derivatives of $f$ belong to $A(\Omega).$ The natural topology on $A^\infty(\Omega)$ is that of uniform convergence on $\prod_{i\in I}\overline\Omega$ of all mixed partial derivatives (of finite order). We shall show that the polynomials are dense in $A^\infty(\Omega).$

Let $f\in A^\infty(\Omega), \varepsilon>0, n\in\N$ and a finite set $\widetilde F\subset I$ be given. We must find a polynomial $P,$ such that $|Df-DP|<\varepsilon$ on $\prod_{i\in I}\overline\Omega_i,$ for each mixed partial derivative, for the variables $z_i, i\in\widetilde F,$ where the derivation with respect to each of the variables is of order $\alpha_i, 0\le\alpha_i\le n.$ Since each function in $A(\Omega)$ can be approximated by its ``restriction" to a finite product of $\overline\Omega_i,$ as we have already mentioned, there exists a finite set $F$ with  $\widetilde F\subset F\subset I$ and a point $\zeta=(\zeta_i)_{i\in I}$ in $\prod_{i\in I}\overline\Omega_i,$ such that $|Df-D\Phi|<\varepsilon/2$ on $\prod_{i\in I}\overline\Omega_i,$ where $\Phi(z_i)_{i\in I}=f(w_i)_{i\in I},$ where $w_i=z_i$ for $i\in F$ and $w_i=\zeta_i$ for $i\in I\setminus F$ and this for each mixed partial derivative $D,$ where the derivatives with respect to the variables $z_i, i\in F$ are of order $\alpha_i, 0\le\alpha_i\le n.$ Since $F$ is a finite set, the function $\Phi$ depends on finitely many variables. By the preceding results, there is a polynomial $P,$ such that $|DP-D\Phi|<\varepsilon/2$ on $\prod_{i\in F}\overline\Omega_i$ for each such mixed partial derivative $D.$ Thus, we find that $|DP-Df|<\varepsilon$ on $\prod_{i\in I}\overline\Omega_i$ for the preceding finite set of partial differential operators. This concludes the proof.

It is also worth noting that, even if $I$ is uncountable, every function in $A^\infty(\Omega)$ depends on only a countable set of variables, because it is continuous with respect to the product topology \cite{MN}.

We present an example. The function
$$
    f((z_i)_{i=1}^\infty) = \sum_{n=1}^\infty\frac{z_n^n}{n^2}
$$
belongs to $A^\infty$ of the polydisc $\prod_{i=1}^\infty\{z:|z|<1\}$ as can easily be verified.
However, its directional derivative at a point $(z_{m,1},z_{m,2},\cdots,z_{m,m},0,0,\cdots),$ in the direction $\nu=(1,1,\cdots)$ equals
$$
    \sum_{n=1}^m \frac{z_{m,n}^{n-1}}{n}
$$
and does not stay bounded, if we choose $z_{m,n}=1-1/m, n=1,2,\cdots,m.$
Thus, the directional derivative does not extend continuously to the closed polydisc of infinitely many variables, for it would be bounded since the product is compact. This is in contrast with the situation in finitely many variables.


\section{Proof of Theorem \ref{chordal}}

Suppose
$$
    \sup_{z\in\overline G}\chi(P_n(z),f(z))\rightarrow 0,
$$
where $P_n(z)$ are polynomials (each one depending on a finite number of variables), $f:\overline G\rightarrow\overline\C$ is a function and $z=(z_i)_{i\in I}\in\overline G, z_i\in \overline G_i$ and $\overline\C=\C\cup\{\infty\}.$

If $f(z)=\infty$ for some point $z=(z_i)_{i\in I}$ in $G,$ then according to \cite{N}, it follows from the Hurwitz theorem in one complex variable that $f$ is identically equal to $\infty$ on every complex line passing through the point $z.$ It follows easily that $f\equiv\infty$ on $\overline G.$ Assume $f(G)\subset\C.$ Then, according to \cite{BGH} and \cite{N}, the function $f$ is holomorphic on each complex line in $G.$ Finally, since every polynomial is continuous on $\overline G,$ endowed with the product topology, the $\chi$-uniform limit $f$ is also continuous on $\overline G.$ This completes the proof of one direction of Theorem \ref{chordal}.

The constant function equal to $\infty$ is the $\chi$-uniform limit of the constant polynomials $P_n\equiv n, n=1,2,\cdots.$ Let $f:\overline G\rightarrow\C$ be continuous, $f(G)\subset\C$ and $f|_G$ separately holomorphic. We consider a conformal mapping $\phi_i:G_i\rightarrow\D$ of $G_i$ onto the open unit disc $\D=\{w\in\C:|w|<1\}.$ According to the Osgood-Carath\'eodory Theorem (see \cite{K}), $\phi_i$ extends to a homeomorphism $\phi_i:\overline G_i\rightarrow\overline\D.$ Let $\phi=(\phi_i)_{i\in I}:\overline G\rightarrow\overline\D^I$ be defined by $\phi(z)=(\phi_i(z_i))_{i\in I}.$ Then, $\phi$ is a homeomorphism, $\phi^{-1}=(\phi_i^{-1})_{i\in I},$ $\phi_i$ is holomorphic on $G_i$ and $\phi_i^{-1}$ is holomorphic on $\D.$ The function $f\circ\phi^{-1}$ is the $\chi$-uniform limit of polynomials $Q_n$ on $\overline\D^I$ (see \cite{MN}). Therefore, we have the $\chi$-uniform convergence of $Q_n\circ\phi$ to $f$ on $\overline G.$ But $Q_n\circ\phi\in A(\overline G).$ In fact, $Q_n\circ\phi$ depends on a finite number of variables. According to Lemma \ref{P(K)=A(K)}, there is a polynomial $P_n$ such that
$$
    \sup_{z\in\overline G}|P_n(z)-(Q_n\circ\phi)(z)|<1/n.
$$
Since $\chi(a,b)\le|a-b|$ for all $a,b\in\C,$ it follows that $\sup\chi(P_n(z),(Q_n\circ\phi)(z))<1/n$ on $\overline G.$ Since,
$$
    \sup_{z\in\overline G}\chi((Q_n\circ\phi)(z),f(z))\rightarrow 0, \quad \mbox{as} \quad n\rightarrow\infty,
$$
the triangle inequality implies that $P_n\rightarrow f$ $\chi$-uniformly on $\overline G.$ This completes the proof.

\end{document}